\documentclass{amsart}

\usepackage{comment}

\newtheorem{theorem}{Theorem}[section]
\newtheorem{lemma}[theorem]{Lemma}
\newtheorem{proposition}[theorem]{Proposition}

\theoremstyle{definition}
\newtheorem{definition}[theorem]{Definition}

\theoremstyle{remark}

\numberwithin{equation}{section}

\begin{document}

\title{Expected First Return Times for Random Walks on Bounded Grids}

\author{Nan An}

\thanks{Date: \today}
\thanks{Email: ann5@mcmaster.ca}

\begin{abstract}
  We derive a general formula for computing the expected first return time of a random walk on a finite graph. Using this framework, we calculate the expected first return time in various settings over bounded rectangular grids with different boundary conditions.
\end{abstract}

\maketitle

\section{Problem Setup}

We follow standard definitions and notation for Markov chains as in \cite{levin}\cite{durrett}.

We study random walks on general finite graphs. Let \( G = (V, E) \) be an graph with vertex set \( V \) and edge set \( E \), where the walk proceeds by moving from a vertex to a uniformly randomly chosen neighbour at each time step. We assume \( G \) is connected and nontrivial. The number of vertices is \( |V| = N \).

We model the walk as a Markov chain with state space \( V \) and transition matrix \( U \in \mathbb{R}^{N \times N} \). Let \( s \in \mathbb{R}^N \) be a probability distribution over \( V \), so that its evolution is given by
\[
s^{k+1} = s^k U.
\]
We fix an arbitrary indexing $f$ of \( V \), so that distributions and matrices can be treated as vectors and matrices in \( \mathbb{R}^N \).

\begin{definition}
Let \( o \in V \) be a fixed starting vertex. For \( k \in \mathbb{N}^+ \), let \( p_k \) be the probability that the walker returns to \( o \) for the first time at step \( k \). The \textit{expected first return time} to \( o \), denoted \( E \), is
\[
E = \sum_{k=1}^\infty k p_k.
\]
\end{definition}

\section{The Waiting Room Construction}

Directly working with return times in \( G \) can be difficult due to cycles and revisits. To isolate first returns to a vertex \( o \), we define a modified Markov chain with a new structure.

\begin{definition}[Waiting Room Construction]
Fix a distinguished vertex \( o \in V \). Define the set of neighbours of \( o \) as
\[
A = \{ x \in V : (x, o) \in E \}.
\]

We construct a modified Markov chain with state space
\[
V^* = (V \setminus \{o\}) \cup \{l, b\},
\]
where:
\begin{itemize}
  \item \( l \) is a \emph{waiting room} state that captures transitions into \( o \),
  \item \( b \) is an \emph{absorbing} state reached from \( l \).
\end{itemize}

The transition matrix \( M \in \mathbb{R}^{(N+1) \times (N+1)} \), where \( N = |V| \), is defined entrywise by:
\[
M_{x, y} =
\begin{cases}
U_{x, y}, & x, y \in V \setminus \{o\}, \\
U_{x, o}, & x \in A,\ y = l, \\
1, & x = l,\ y = b, \\
1, & x = b,\ y = b, \\
0, & \text{otherwise}.
\end{cases}
\]

We define a flattening map \( f^* : V^* \to \{0, \dotsc, N\} \) for indexing the matrix \( M \) as follows:
\[
f^*(x) =
\begin{cases}
f(x), & x \in V \setminus \{o\},\ f(x) < f(o), \\
f(x) - 1, & x \in V \setminus \{o\},\ f(x) > f(o), \\
N - 1, & x = l, \\
N, & x = b.
\end{cases}
\]

With this indexing, the matrix \( M \) takes block form:
\[
M = \begin{bmatrix}
Q & R \\
0 & T
\end{bmatrix},
\]
where
\begin{itemize}
  \item \( Q \in \mathbb{R}^{(N-1) \times (N-1)} \) governs transitions on \( V \setminus \{o\} \),
  \item \( R \in \mathbb{R}^{(N-1) \times 2} \) captures transitions from \( A \) to the waiting room,
  \item \( T \in \mathbb{R}^{2 \times 2} \) handles transitions from \( l \) to \( b \), and absorption at \( b \).
\end{itemize}
\end{definition}

\begin{proposition}
  The matrices \(Q\), \(R\), and \(T\) arising from the "waiting room" construction satisfy:
  \begin{enumerate}
    \item \(Q \in \mathbb{R}^{(n-1)\times(n-1)}\) is a sub-stochastic matrix: each row sum is at most 1. In particular, for \(x \in A\), the row sum of \(Q\) is strictly less than 1. Its spectral radius \(\rho(Q)\) is strictly less than 1.
    \item The only nonzero entries of \(R \in \mathbb{R}^{(n-1)\times 2}\) are in the first column, and are given by \(R_{x,0} = U_{x,o}\). The second column of \(R\) is zero: no state transitions directly to the absorbing state \(b\).
    \item The matrix \(T \in \mathbb{R}^{2\times 2}\) is
    \[
    T = \begin{bmatrix}
    0 & 1 \\
    0 & 1
    \end{bmatrix}, \quad \text{and } T^2 = T.
    \]
  \end{enumerate}
  \end{proposition}
  
  \begin{proof}
  Most of the proposition can be verified directly. As for the spectral radius, we know \(Q\) is irreducible because the graph \(G\) is connected. The claim then follows from the following lemma.
  \end{proof}
  
  \begin{lemma}
  The spectral radius of a sub-stochastic irreducible matrix is strictly less than 1.
  \end{lemma}
  
  \begin{proof}
  Suppose \( A \in \mathbb{R}^{N \times N} \) is such a matrix, meaning it is non-negative and the sum of each row is at most 1. Let \( \lambda = \rho(A) \) be the Perron-Frobenius eigenvalue of \( A \), with corresponding eigenvector \( \nu \in \mathbb{R}^N \), normalized such that \( \|\nu\|_1 = 1 \). By the Perron-Frobenius theorem for irreducible non-negative matrices, all entries of \( \nu \) are strictly positive.
  
  Using this, we compute:
  \[
  |\lambda| = \|\lambda \nu\|_1 = \|\nu A\|_1 = \sum_j \sum_k \nu_j A_{jk}.
  \]
  
  Define
  \[
  \epsilon_j = \frac{1}{N} \left( 1 - \sum_{k=1}^N A_{jk} \right),
  \]
  so that adding \( \epsilon_j \) to each element of the \( j \)-th row of \( A \) yields a row sum of 1. Let \( \epsilon \in \mathbb{R}^N \) be the row vector of the \( \epsilon_j \)'s. Then:
  \[
  |\lambda| = \sum_j \sum_k \nu_j \left(A_{jk} + \epsilon_j - \epsilon_j \right)
  = \sum_j \sum_k \nu_j (A_{jk} + \epsilon_j) - \sum_j \sum_k \nu_j \epsilon_j.
  \]
  
  Note that:
  \[
  \sum_j \sum_k \nu_j \epsilon_j = \sum_j \nu_j \cdot N \epsilon_j = N (\epsilon \cdot \nu),
  \]
  so:
  \[
  |\lambda| = \left\| \nu (A + \epsilon^T \mathbf{1}) \right\|_1 - N (\epsilon \cdot \nu).
  \]
  
  Define the matrix \( \widehat{A} = A + \epsilon^T \mathbf{1} \in \mathbb{R}^{N \times N} \), where \( \mathbf{1} \in \mathbb{R}^N \) is the column vector of ones. Then \( \widehat{A} \) is a proper stochastic matrix (each row sums to 1). Since \( \nu > 0 \) and \( \epsilon \geq 0 \) with at least one \( \epsilon_j > 0 \), we have \( \epsilon \cdot \nu > 0 \). Therefore:
  \[
  |\lambda| = \left\| \nu \widehat{A} \right\|_1 - N (\epsilon \cdot \nu) = 1 - N (\epsilon \cdot \nu) < 1.
  \]
  
  Thus, the spectral radius \( \lambda = \rho(A) \) is strictly less than 1.
\end{proof}

\begin{proposition}
  For all \( k \geq 1 \), the powers of the matrix \( M \) satisfy:
  \[
  M^k =
  \begin{bmatrix}
  Q^k & \sum_{j=0}^{k-1} Q^j R T^{k-1-j} \\
  0 & T
  \end{bmatrix}.
  \]
  \end{proposition}
  
  \begin{proof}[Sketch of Proof]
  This can be shown by induction on \( k \). The base case \( k = 1 \) is immediate from the block structure of \( M \). Suppose the formula holds for \( k \), then for \( k+1 \):
  \begin{align*}
    M^{k+1} = M^k M &=
  \begin{bmatrix}
  Q^k & \sum_{j=0}^{k-1} Q^j R T^{k-1-j} \\
  0 & T
  \end{bmatrix}
  \begin{bmatrix}
  Q & R \\
  0 & T
  \end{bmatrix}\\
  &=
  \begin{bmatrix}
  Q^{k+1} & \sum_{j=0}^{k-1} Q^{j} R T^{k - j} + Q^k R T^0 \\
  0 & T^2
  \end{bmatrix},
  \end{align*}
  which simplifies to the claimed formula by reindexing the sum.
  \end{proof}
  
  \begin{proposition}
    Let \( Q \in \mathbb{R}^{n \times n} \) be a matrix such that \( \rho(Q) < 1 \). Then the following matrix series converge entry-wise:
    \[
    S_1 = \sum_{k=0}^\infty Q^k, \quad S_2 = \sum_{k=1}^\infty k Q^{k-1}.
    \]
    Moreover,
    \[
    S_1 = (I - Q)^{-1}, \quad S_2 = (I - Q)^{-2}.
    \]
    \end{proposition}
    
    \begin{proof}
    Convergence follows from the fact that, since \( \rho(Q) < 1\), \(\|Q^k\|_2 \to 0 \) exponentially fast, hence \( \sum_{k=0}^\infty \|Q^k\|_2 < \infty \) and \( \sum_{k=1}^\infty k \|Q^{k-1}\|_2 < \infty \). This implies absolute convergence in norm, and thus entry-wise convergence of the matrix series \( S_1 \) and \( S_2 \).
    
    To compute \( S_1 \), observe:
    \[
    S_1 - Q S_1 = \sum_{k=0}^\infty Q^k - \sum_{k=0}^\infty Q^{k+1} = I,
    \]
    so \( S_1 = (I - Q)^{-1} \).
    
    For \( S_2 \), we consider the square of the geometric series:
    \[
    S_1^2 = \left( \sum_{i=0}^\infty Q^i \right) \left( \sum_{j=0}^\infty Q^j \right) = \sum_{k=0}^\infty \left( \sum_{i=0}^k Q^i Q^{k-i} \right) = \sum_{k=0}^\infty (k+1) Q^k = S_2
    \]
    This follows by grouping terms where \( i + j = k \).
    \end{proof}

    \begin{theorem}
      Let \( U \in \mathbb{R}^{n \times n} \) be the transition matrix of a Markov chain corresponding to a random walk on a graph, with a distinguished origin state \( o \). Let \( \hat{s} \in \mathbb{R}^n \) be the row of \( U \) corresponding to \( o \), and let \( \hat{s'} \in \mathbb{R}^n \) be the column of \( U \) corresponding to \( o \). Let \( s \), \( s' \) denote their extensions under the waiting room construction (i.e., where the entry corresponding to \( o \) is removed).
      
      Let \( Q \) be the submatrix of the modified transition matrix \( M \) corresponding to the transient states (i.e., all states except the origin), so that \( \rho(Q) < 1 \), and let
      \(
      N = (I - Q)^{-1}.
      \)
      
      Then, the expected first return time to $o$ is given by:
      
      $$
      E = U_{o,o} + s N^2 (s')^T = U_{o,o} + (sN)(s'N^T)^T,
      $$
    \end{theorem}
    \begin{proof}
      Since one step has already occurred from the origin, the return time is $1\cdot P_0=U_{o,o}$ plus the expected number of additional steps until absorption in the left waiting room state \( l \). Let \( s^k \in \mathbb{R}^{1 \times (n+1)} \) denote the distribution after \( k \) steps in the waiting room Markov chain, starting from \( s^0 \). Then:
      
      \begin{align*}
        E &= U_{o,o} + \sum_{k=1}^\infty k \cdot (s^k)_l \\
        &= U_{o,o} + \sum_{k=1}^\infty k \cdot (s M^k)_{l}\\
        &= U_{o,o} + \sum_{k=1}^\infty k \sum_{a \in A} U_{o,a} (M^k)_{a,l}
      \end{align*}
      
      By the block structure of \( M^k \), and noting that \(T\) kills contributions to the first column, we have
      
      \[
      (M^k)_{a,l} = (Q^{k-1} R)_{a,0},
      \]
      
      Hence,

      \begin{align*}
        E &= U_{o,o} + \sum_{k=1}^\infty k \sum_{a,a' \in A} U_{o,a} (Q^{k-1})_{a,a'}U_{a',o} \\
        & = U_{o,o} + s N^2 (s')^T
      \end{align*}
      
    \end{proof}

    This result is useful both analytically and numerically. As an example, we will compute the first return time for some cases on rectangular bounded grids.
      
    \section{Over Bounded Grids}

    We consider a random walk on a bounded rectangular grid of dimension $d$, with side lengths $n_1, \dots, n_d$. That is, the state space is the discrete box
    
    $$
    V = \{0,1,\dots,n_1-1\} \times \cdots \times \{0,1,\dots,n_d-1\} \subset \mathbb{Z}^d.
    $$
    
    At each step, the walker selects one of the directions uniformly at random and attempts to move one unit in that direction. We investigate several standard boundary conditions:

    \begin{enumerate}
      \item Periodic (torus): The grid wraps around; stepping off one side places the walker on the opposite side.
      \item Stay-still (absorbing wall): If a move would take the walker outside the grid, it is canceled — the walker stays in place.
      \item Reflecting: If a move would cross the boundary, the walker instead moves one unit in the opposite direction (bounces back).
    \end{enumerate}

\begin{proposition}
  For case (1), for any point on the grid, the expected return time is $n_1n_2\cdots n_d$. For (2), for internal points, it is $n_1n_2\cdots n_d$. For points on the boundary, it is $1/4+n_1n_2\cdots n_d$, while on the cornor it is $1/2 + n_1n_2\cdots n_d$. For (3), the expected return time, for internal points, is $(n_1-1)(n_2-1)\cdots(n_d-1)$; for edge points, is $2(n_1-1)(n_2-1)\cdots(n_d-1)$; for cornor points, it is $4(n_1-1)(n_2-1)\cdots(n_d-1)$.
\end{proposition}

\begin{proof}
  For case (1) and (2), since $Q$ is symmetric, we know that $s=s'$, $N$ is symmetric, and $s=\mathbf{1}(I-Q)$, where \( \mathbf{1} \in \mathbb{R}^N \) is the row vector of ones. Therefore, the second term of $E$ simplifies to the number of states.

  For case (3), $Q$ is not symmetric anymore. we still have $s'=\mathbf{1}(I-Q)^T$. However, we need to solve $s=x(I-Q)$. Define $y$ by
  $$
  y_w=\begin{cases}
    1/4, &\text{if $w$ is on the corner}\\
    1/2, &\text{if $w$ is on the edge}\\
    1, &\text{if $w$ is internal}
  \end{cases}
  $$
  We observe that when $o$ is internal, $x=y$. When $o$ is on the edge, $o=2y$. When $o$ is on the corner, $o=4y$. Taking the sum over the components gives us the result.
\end{proof}

\end{document}